\documentclass[10pt,a4paper]{amsart}
\usepackage[a4paper,marginratio=1:1]{geometry}
\usepackage[protrusion=all]{microtype}
\usepackage{amssymb}
\usepackage[initials,non-sorted-cites]{amsrefs}
\usepackage{mybibspec}
\usepackage[shortlabels]{enumitem}
\usepackage[symbol,perpage]{footmisc}
\usepackage{graphicx}
\usepackage{braket,tensor,bm,stmaryrd,fouridx}

\numberwithin{equation}{section}
\usepackage{mathtools}\mathtoolsset{showonlyrefs}

\newtheorem{thm}{Theorem}[section]
\newtheorem{prop}[thm]{Proposition}

\theoremstyle{definition}
\newtheorem{dfn}[thm]{Definition}
\theoremstyle{remark}
\newtheorem{rem}[thm]{Remark}

\newcommand{\abs}[1]{\lvert#1\rvert}
\newcommand{\Abs}[1]{\left\lvert#1\right\rvert}
\newcommand{\norm}[1]{\lVert#1\rVert}
\newcommand{\bdry}{\partial}
\newcommand{\orddot}{\mathord{\cdot}}
\newcommand{\compose}{\mathbin{\circ}}
\newcommand{\bdryX}{{\bdry X}}
\newcommand{\cornerX}{{\angle X}}

\DeclareMathOperator{\tr}{tr}
\DeclareMathOperator{\ran}{ran}
\DeclareMathOperator{\Hom}{Hom}
\DeclareMathOperator{\rank}{rank}

\title[Harmonic maps from the product of the hyperbolic planes]{Harmonic maps from the product of the hyperbolic planes to the hyperbolic space}
\author[K. Akutagawa]{Kazuo Akutagawa}
\address{Department of Mathematics, Chuo University, Bunkyo-ku, Tokyo 112-8551, Japan}
\email{akutagawa@math.chuo-u.ac.jp}
\author[Y. Matsumoto]{Yoshihiko Matsumoto}
\address{Department of Mathematics, Graduate School of Science, The University of Osaka, Toyonaka, Osaka 560-0043, Japan}
\email{matsumoto.yoshihiko.sci@osaka-u.ac.jp}

\begin{document}

\maketitle

\begin{abstract}
	An existence result is shown for the asymptotic Dirichlet problem for
	harmonic maps from the product of the hyperbolic planes to the hyperbolic space,
	where the Dirichlet data is given on the distinguished boundary (the product of the circles at infinity).
\end{abstract}

\section*{Introduction}

This paper concerns harmonic maps with unbounded images between noncompact manifolds.
The first significant work regarding them appeared as the theory of the Gauss maps of
spacelike constant mean curvature hypersurfaces in the Minkowski spaces
by Choi and Treibergs \cites{Choi-Treibergs-88,Choi-Treibergs-90,Choi-Treibergs-93}.
Shortly thereafter, the asymptotic Dirichlet problem for such harmonic maps between the hyperbolic spaces was
investigated in a series of papers of Li and Tam \cites{Li-Tam-91,Li-Tam-93-I,Li-Tam-93-II} and,
for the hyperbolic disks, by the first author \cite{Akutagawa_94}.
Analogous results for harmonic maps
between other pairs of rank-one symmetric spaces of noncompact type was obtained by
Donnelly \cite{Donnelly-94-harmonic}.
More recently, the present authors studied the existence and uniqueness of
proper harmonic maps between asymptotically hyperbolic manifolds \cite{Akutagawa-Matsumoto-16}.

Despite these developments, various aspects of harmonic maps with unbounded images
remain to be clarified.
A natural next step is to investigate such harmonic maps between higher-rank symmetric spaces
of noncompact type,
and we shall focus here on the asymptotic Dirichlet problem for harmonic maps
from the product $\mathbb{H}^{m_1+1}\times\mathbb{H}^{m_2+1}$ of two hyperbolic spaces
into the hyperbolic space $\mathbb{H}^{n+1}$, where $m_1$, $m_2$, and $n$ are positive integers.
Although the product of two hyperbolic spaces is not an irreducible symmetric space,
it serves as a local model for certain higher-rank irreducible symmetric spaces of noncompact type.
There is no need to consider a product as the target,
since if $N=N_1\times N_2$ is a Riemannian product,
any harmonic map into $N$ splits as $u=(u_1,u_2)$, with
each $u_\lambda$ being a harmonic map into $N_\lambda$.

As an initial step in this direction, we establish the following theorem in this paper.

\begin{thm}
	\label{thm:main}
	Let $D$ and $B^{n+1}$ be the unit disk in $\mathbb{R}^2$ and the unit ball in $\mathbb{R}^{n+1}$,
	both equipped with the Poincar\'e metrics.
	Then, for any $C^\infty$-smooth map $\varphi\colon\bdry D\times\bdry D\to\bdry B^{n+1}$
	with nowhere vanishing factorwise differential,
	there exists a harmonic map $u\in C^\infty(D\times D,B^{n+1})$ that extends
	to a continuous map $u\in C^0(\overline{D\times D},\overline{B^{n+1}})$ satisfying
	$u|_{\bdry D\times\bdry D}=\varphi$ and
	$u(\bdry D\times D\cup D\times\bdry D)\subset B^{n+1}$.
\end{thm}

Here, we say that $\varphi$ has nowhere vanishing factorwise differential
if $d_1\varphi$ and $d_2\varphi$ are both nowhere vanishing,
where $d\varphi=d_1\varphi+d_2\varphi$ is the decomposition corresponding to
\begin{equation}
	T^*_{(q_1,q_2)}(\bdry D\times\bdry D)\otimes T_{\varphi(q_1,q_2)}\bdry B^{n+1}
	=(T^*_{q_1}\bdry D\otimes T_{\varphi(q_1,q_2)}\bdry B^{n+1})
	\oplus(T^*_{q_2}\bdry D\otimes T_{\varphi(q_1,q_2)}\bdry B^{n+1}).
\end{equation}

To this end, we first make a few preliminary observations in Section \ref{sec:prelim}.
In Section \ref{sec:smooth-dependence-on-boundary-values}, we study the dependence of
harmonic maps between hyperbolic spaces discussed in \cite{Li-Tam-93-I,Li-Tam-93-II}
on smooth families of boundary data.
This is then used as an auxiliary result for the construction of approximate solutions to our problem
in Section \ref{sec:approx-solution}.
In the last section, we show Theorem \ref{thm:main}
using results in \cite{Li-Tam-91} on the heat equation for harmonic maps.

The idea of prescribing the boundary value only along the corner $\bdry D\times\bdry D$ of the domain
is motivated by the fact that a holomorphic function in an open polydisk in $\mathbb{C}^m$
that admits continuous extension to the closure is uniquely determined by its restriction to
the distinguished boundary (the Shilov boundary).
Proposition \ref{prop:individual-harmonicity} will strengthen the reason why we should consider
such restricted boundary values.
Furthermore, note that typical examples of harmonic maps $u\colon D\times D\to D$ are given by
the holomorphic maps
\begin{equation}
	D\times D\to D,\qquad (z,w)\mapsto z^kw^l,
\end{equation}
where $k$, $l\in\mathbb{N}$.
The condition $u(\bdry D\times D\cup D\times\bdry D)\subset B^{n+1}$ in the statement of our theorem
is incorporated in view of these examples.
So to speak, the harmonic maps that we construct are ``proper only at the corner.''

The dimensional restriction $m_1=m_2=1$ in our main theorem is technically indispensable for us.
Indeed, Proposition \ref{prop:dimension-restriction} shows that
harmonic maps $u\colon B^{m_1+1}\times B^{m_2+1}\to B^{n+1}$ satisfying the conditions as in
Theorem \ref{thm:main} that belong to
$C^2(\overline{B^{m_1+1}\times B^{m_2+1}},\overline{B^{n+1}})$
do not exist unless $m_1=m_2=1$.
While we do not assert such $C^2$-regularity for the harmonic maps in Theorem \ref{thm:main},
this fact suggests that the cases with higher-dimensional domains require
substantially more intricate considerations.
This is why we confine ourselves to the lowest-dimensional setting.
Furthermore, our proof of Theorem \ref{thm:smooth-dependence-of-harmonic-extension-on-hyperbolic-space},
which is given based on techniques from the theory of conformally compact Einstein metrics,
also relies heavily on the assumption $m_1=m_2=1$.

The both authors have greatly benefited by many discussions with Rafe Mazzeo.
The first and second authors are supported in part by JSPS KAKENHI Grant Numbers 24K06718 and 24K06738, respectively.

\section{Preliminary observations}
\label{sec:prelim}

In this section, we write $X=\mathbb{H}^{m_1+1}\times\mathbb{H}^{m_2+1}$,
where $m_1$ and $m_2$ are positive integers.
The space $X$ is equipped with the product metric of the hyperbolic metrics, which is denoted by $g$.
Each hyperbolic space $\mathbb{H}^{m_\lambda+1}$, $\lambda=1$, $2$, will be identified
with the Poincar\'e unit ball $B^{m_\lambda+1}$.
The upper-half space model will also be used for computational purposes.

We write
\begin{equation}
	\overline{X}=\overline{B^{m_1+1}}\times\overline{B^{m_2+1}}
\end{equation}
and
\begin{equation}
	\bdryX=(\bdry B^{m_1+1}\times\overline{B^{m_2+1}})\cup(\overline{B^{m_1+1}}\times\bdry B^{m_2+1}).
\end{equation}
The latter will be referred to
as the \emph{topological boundary} of $\overline{X}$ to emphasize its contrast to
the \emph{distinguished boundary} or the \emph{corner}
\begin{equation}
	\cornerX=\bdry B^{m_1+1}\times\bdry B^{m_2+1}.
\end{equation}

We are interested in studying harmonic maps from $X$
targeted into $\mathbb{H}^{n+1}$, identified with $B^{n+1}$.
More specifically, we want to construct a harmonic map $u\colon X\to B^{n+1}$
that can be (at least) continuously extended to a map $u\colon\overline{X}\to\overline{B^{n+1}}$
that assumes a given value $\varphi\colon \cornerX\to\bdry B^{n+1}$ along the corner.
Recall that a $C^2$ map $u\colon X\to B^{n+1}$ is said to be \emph{harmonic}
if its tension field $\tau(u)$ vanishes,
where $\tau(u)$ is the section of $u^*TB^{n+1}$ defined by $\tau(u)=-\nabla^*du$.

We impose the following ``properness only at the corner'' condition on harmonic maps $u$ that we seek:
\begin{equation}
	\label{eq:properness-at-corner}
	u(\cornerX)\subset\bdry B^{n+1},\qquad u(\overline{X}\setminus\cornerX)\subset B^{n+1}.
\end{equation}

Regarding the corner value $\varphi\colon\cornerX\to\bdry B^{n+1}$,
which we assume to be at least $C^1$, we introduce an assumption that
\begin{equation}
	\label{eq:nondegeneracy}
	\text{$d_1\varphi$ and $d_2\varphi$ are both nowhere vanishing.}
\end{equation}

In what follows, we will sometimes need the energy density $e(\varphi)=\abs{d\varphi}^2$ of
$\varphi\colon\cornerX\to\bdry B^{n+1}$.
In order to define this,
we can either use the standard round metric of the spheres $\bdry B^{m_\lambda+1}$ and $\bdry B^{n+1}$,
or locally, via the upper-half space model,
the Euclidean metric on $\bdry H_+^{m_\lambda+1}$ and on $\bdry H_+^{n+1}$.
Whichever metric we use, there is a decomposition
\begin{equation}
	\label{eq:energy-density-decomposition}
	e(\varphi)=e_1(\varphi)+e_2(\varphi)
\end{equation}
corresponding to the decomposition $d\varphi=d_1\varphi+d_2\varphi$
of the differential into the $T^*\bdry\mathbb{H}^{m_1+1}$ and $T^*\bdry\mathbb{H}^{m_2+1}$ components.

For computations of the tension field of $u$,
let us trivialize $TB^{n+1}$ and thus $u^*TB^{n+1}$ using the upper-half space coordinates of $\mathbb{H}^{n+1}$.
That is, instead of $B^{n+1}$, we identify $\mathbb{H}^{n+1}$ with the upper-half space
\begin{equation}
	H_+^{n+1}=\set{(\xi,\eta^1,\dots,\eta^n)|\xi>0}\subset\mathbb{R}^{n+1},
\end{equation}
and express the mapping $u$ as $(u^0,u^1,\dots,u^n)$
and the tension field $\tau(u)$ as $(\tau^0(u),\tau^1(u),\dots,\tau^n(u))$ accordingly,
where
\begin{equation}
	\tau(u)=\tau^0(u)\partial_\xi+\sum_{i=1}^n\tau^i(u)\partial_{\eta^i}.
\end{equation}
We will also abbreviate $(\tau^0(u),\tau^1(u),\dots,\tau^n(u))$ as $(\tau^0,\tau^1,\dots,\tau^n)$.

Let $(x_\lambda,y_\lambda)=(x_\lambda,y_\lambda^1,\dots,y_\lambda^{m_\lambda})$ be
the standard coordinates in the upper-half space model
of the factor $\mathbb{H}^{m_\lambda+1}$ in $X$.
Then, if we write $\overline{g}_\lambda$ for the Euclidean metric on $H^{m_\lambda+1}_+$,
the components of the tension field are given by the following formula:
\begin{subequations}
	\label{eq:tension-field}
\begin{align}
	\label{eq:tension-field-normal}
	\frac{1}{u^0}\tau^0
	&=\sum_{\lambda=1}^{2}
	x_\lambda^2
	\left(\frac{1}{u^0}\Delta_{\overline{g}_\lambda}u^0
	-\frac{m_\lambda-1}{x_\lambda u^0}\partial_{x_\lambda}u^0
	-\frac{1}{(u^0)^2}\left(\abs{\nabla_{\overline{g}_\lambda}u^0}_{\overline{g}_\lambda}^2
	-\sum_{i=1}^n\abs{\nabla_{\overline{g}_\lambda}u^i}_{\overline{g}_\lambda}^2\right)\right),\\
	\label{eq:tension-field-tangential}
	\frac{1}{u^0}\tau^i
	&=\sum_{\lambda=1}^{2}
	x_\lambda^2
	\left(\frac{1}{u^0}\Delta_{\overline{g}_\lambda}u^i
	-\frac{m_\lambda-1}{x_\lambda u^0}\partial_{x_\lambda}u^i
	-\frac{1}{(u^0)^2}
	\braket{\nabla_{\overline{g}_\lambda}u^0,\nabla_{\overline{g}_\lambda}u^i}_{\overline{g}_\lambda}\right).
\end{align}
\end{subequations}
Here, $\Delta_{\overline{g}_\lambda}$ denotes the Euclidean Laplacian:
\begin{equation}
	\Delta_{\overline{g}_\lambda}=
	\partial_{x_\lambda}^2+\partial_{y_\lambda^1}^2+\dots+\partial_{y_\lambda^{m_\lambda}}^2.
\end{equation}

We begin with the following observation concerning the ``individual harmonicity on $\bdryX$''
of harmonic maps, or more generally, of mappings whose tension field vanishes along $\bdryX$.

\begin{prop}
	\label{prop:individual-harmonicity}
	Let $u\colon\overline{X}\to\overline{B^{n+1}}$ be a continuous map
	satisfying the condition \eqref{eq:properness-at-corner} such that
	\begin{enumerate}[(i)]
		\item
			\label{item:twice-differentiability-up-to-boundary-off-corner}
			$u\in C^2(\overline{X}\setminus\cornerX,B^{n+1})$, and
		\item
			\label{item:tension-field-vanishing-at-boundary}
			The pointwise norm $\abs{\tau(u)}$ of the tension field
			tends to zero uniformly at $\bdryX$.
	\end{enumerate}
	Then, $u|_\bdryX$ must be individually harmonic for each variable,
	i.e., the mapping
	\begin{equation}
		u_1=u(\orddot,q_2)\colon B^{m_1+1}\to B^{n+1}\qquad
		\text{(resp.\ $u_2=u(q_1,\orddot)\colon B^{m_2+1}\to B^{n+1}$)}
	\end{equation}
	is harmonic for any $q_2\in\bdry B^{m_2+1}$ (resp.\ for any $q_1\in\bdry B^{m_1+1}$).
\end{prop}

\begin{proof}
	It suffices to show that $u_1$ is harmonic.
	Let $q_2\in\bdry B^{m_2+1}$, and we also fix $p_1\in B^{m_1+1}$.
	Instead of the Poincar\'e ball models,
	take the upper-half space models of the respective hyperbolic spaces in such a way that
	$q_2$ lies on the hyperplane $x_2=0$,
	and express the tension field as $\tau(u)=(\tau^0(u),\tau^1(u),\dots,\tau^n(u))$.
	Then, as $p_2\in\mathbb{H}^{m_2+1}$ tends to $q_2$,
	the $\lambda=2$ term in the right-hand side of \eqref{eq:tension-field-normal}
	evaluated at $(p_1,p_2)$ converges to zero.
	Consequently, $(u^0(p_1,p_2))^{-1}\tau^0(p_1,p_2)$ tends to the $\lambda=1$ term evaluated at $(p_1,q_2)$.
	If we write the tension field of $u_1=u(\orddot,q_2)$ as
	$\tau(u_1)=(\tau^0(u_1),\tau^1(u_1),\dots,\tau^n(u_1))$, the above observation shows that
	\begin{equation}
		\label{eq:limit-in-individual-harmonicity-proof}
		\lim_{p_2\to q_2}\frac{1}{u^0(p_1,p_2)}\tau^0(u)(p_1,p_2)
		=\frac{1}{u^0(p_1,q_2)}\tau^0(u_1)(p_1).
	\end{equation}
	Since
	\begin{equation*}
		\abs{\tau(u)}^2=\left(\frac{1}{u^0}\tau^0(u)\right)^2
		+\sum_{i=1}^n\left(\frac{1}{u^0}\tau^i(u)\right)^2,
	\end{equation*}
	our assumption (ii) implies that the limit \eqref{eq:limit-in-individual-harmonicity-proof} must be zero,
	which means that $\tau^0(u_1)(p_1)$ vanishes.
	Similar discussion using \eqref{eq:tension-field-tangential} implies that $\tau^i(u_1)(p_1)=0$.
\end{proof}

Recall from \cite[Lemma 1.3]{Li-Tam-93-I} that,
if $u\in C^\infty(B^{m+1},B^{n+1})$ is a proper harmonic mapping that extends to
$u\in C^1(\overline{B^{m+1}},\overline{B^{n+1}})$
such that $\varphi=u|_{\bdry B^{m+1}}$ has nowhere vanishing differential, then
\begin{equation}
	\label{eq:Li-Tam-lemma}
	\left.\frac{\partial u^0}{\partial x}\right|_{x=0}=\sqrt{\frac{e(\varphi)}{m}}
	\qquad\text{and}\qquad
	\left.\frac{\partial u^i}{\partial x}\right|_{x=0}=0
\end{equation}
with respect to the upper-half space coordinates.
(Although originally stated with respect to the polar coordinates, the result remains true
for the upper-half space coordinates;
see also Akutagawa--Matsumoto \cite{Akutagawa-Matsumoto-16}*{Lemma 4}.)
In view of this, the following is immediate.

\begin{prop}
	\label{prop:Li-Tam-asymptotics}
	Let $u\colon\overline{X}\to\overline{B^{n+1}}$ be a continuous map
	satisfying \eqref{eq:properness-at-corner},
	and assume that $\varphi=u|_\cornerX$ is a $C^1$ map satisfying \eqref{eq:nondegeneracy}.
	Suppose also that \ref{item:twice-differentiability-up-to-boundary-off-corner},
	\ref{item:tension-field-vanishing-at-boundary} in Proposition \ref{prop:individual-harmonicity} and
	\begin{enumerate}[(i), start=3]
		\item
			\label{item:once-differentiability-up-to-topological-boundary}
			$u\in C^1(\overline{X},\overline{B^{n+1}})$
	\end{enumerate}
	are satisfied.
	Then locally, with respect to the upper-half space coordinates,
	\begin{equation}
		\left.\frac{\partial u^0}{\partial x_\lambda}\right|_{x_1=x_2=0}=\sqrt{\frac{e_\lambda(\varphi)}{m_\lambda}}
		\qquad\text{and}\qquad
		\left.\frac{\partial u^i}{\partial x_\lambda}\right|_{x_1=x_2=0}=0
	\end{equation}
	for $\lambda=1$ and $2$,
	where $e_\lambda(\varphi)$ is the partial energy density in \eqref{eq:energy-density-decomposition}
	defined by the Euclidean metrics on $\bdry H^{m_\lambda+1}_+$ and $\bdry H^{n+1}_+$.
\end{prop}

The above fact on the asymptotic behavior leads to the following uniqueness result.

\begin{thm}
	\label{thm:uniqueness}
	Let $u$ and $\tilde{u}$ be two harmonic maps from $X$ to $B^{n+1}$
	satisfying all the assumptions in Proposition \ref{prop:Li-Tam-asymptotics},
	and moreover we assume that $u|_{\cornerX}=\tilde{u}|_{\cornerX}$.
	Then $u=\tilde{u}$ in $X$.
\end{thm}

\begin{proof}
	It follows from Proposition \ref{prop:individual-harmonicity} and
	\cite{Li-Tam-93-I}*{Theorem 1.1} that $u|_{\bdryX}=\tilde{u}|_{\bdryX}$.
	By this fact and Proposition \ref{prop:Li-Tam-asymptotics},
	if we define $\psi\colon X\to B^{n+1}\times B^{n+1}$ by $\psi(p)=(u(p),\tilde{u}(p))$,
	the function $d_{\mathbb{H}^{n+1}}\compose\psi$ continuously extends to $\overline{X}$ and vanishes on $\bdryX$,
	where $d_{\mathbb{H}^{n+1}}$ is the hyperbolic distance on $B^{n+1}$.
	Since $u$ and $\tilde{u}$ are harmonic, we have
	\begin{equation}
		\Delta_g(d_{\mathbb{H}^{n+1}}\compose\psi)=\tr_g\psi^*\nabla^2d_{\mathbb{H}^{n+1}},
	\end{equation}
	and the right-hand side is non-negative by the computations in \cite{Jager-Kaul-79}.
	It follows from the maximum principle that $d_{\mathbb{H}^{n+1}}\compose\psi$ equals zero identically.
	(One can also argue by the fact that $(d_{\mathbb{H}^{n+1}}\compose\psi)^2$ is a subharmonic function,
	which follows from \cite{Schoen-Yau-79}*{Proposition 1}.)
\end{proof}

We close this section with the following notable observation.

\begin{prop}
	\label{prop:dimension-restriction}
	Suppose that $u\colon\overline{X}\to\overline{B^{n+1}}$ satisfies all the assumptions in
	Proposition \ref{prop:Li-Tam-asymptotics}, with \ref{item:once-differentiability-up-to-topological-boundary}
	strengthened to
	\begin{enumerate}[(i$^\sharp$), start=3]
		\item
			\label{item:twice-differentiability-up-to-topological-boundary}
			$u\in C^2(\overline{X},\overline{B^{n+1}})$.
	\end{enumerate}
	Then $m_1=m_2=1$.
\end{prop}

\begin{proof}
	By Proposition \ref{prop:Li-Tam-asymptotics},
	if we locally set $a_\lambda(y_\lambda)=\sqrt{e_\lambda(\varphi)(y_\lambda)/m_\lambda}$
	for $y_\lambda\in\bdry H^{m_\lambda+1}_+$, then
	$u=(u^0,u^1,\dots,u^n)$ has the following asymptotic behavior as $x=(x_1,x_2)\to (0,0)$:
	\begin{equation}
		u^0(x_1,y_1,x_2,y_2)=a_1(y_1)x_1+a_2(y_2)x_2+o(\abs{x}),\qquad
		u^i(x_1,y_1,x_2,y_2)=o(\abs{x}).
	\end{equation}
	Let us fix $y_1$, $y_2$ and take fixed $\gamma_1$, $\gamma_2>0$,
	and examine the behavior of the right-hand side of \eqref{eq:tension-field-normal}
	along the half-line $t\mapsto (t\gamma_1,y_1,t\gamma_2,y_2)$ when $t\to +0$.
	Since $\Delta_{\overline{g}_\lambda}u^0$ is bounded
	because of assumption \ref{item:twice-differentiability-up-to-topological-boundary},
	it converges to
	\begin{equation}
		\label{eq:tension-field-of-harmonic-map}
		-\sum_{\lambda=1}^2\left(\frac{(m_\lambda-1)\gamma_\lambda}{a_1\gamma_1+a_2\gamma_2}\cdot a_\lambda
		+\frac{\gamma_\lambda^2}{(a_1\gamma_1+a_2\gamma_2)^2}(a_\lambda^2-e_\lambda(\varphi))\right),
	\end{equation}
	and by assumption \ref{item:tension-field-vanishing-at-boundary} this must be zero.
	One can check that \eqref{eq:tension-field-of-harmonic-map} equals
	\begin{equation*}
		(2-m_1-m_2)\frac{a_1a_2\gamma_1\gamma_2}{(a_1\gamma_1+a_2\gamma_2)^2}.
	\end{equation*}
	Therefore, $2-m_1-m_2$ has to be zero, which means that $m_1=m_2=1$ must be met.
\end{proof}

\section{Boundary value dependence of proper harmonic maps from $\mathbb{H}^2$ to $\mathbb{H}^{n+1}$}
\label{sec:smooth-dependence-on-boundary-values}

We observed in Proposition \ref{prop:individual-harmonicity} that,
under certain assumptions, the individual harmonicity can be shown for
the (topological) boundary value of harmonic maps $X=B^{m_1+1}\times B^{m_2+1}\to B^{n+1}$
extendible to $\overline{X}\to\overline{B^{n+1}}$.
Although we do not claim that the harmonic map we construct in our proof of Theorem \ref{thm:main} satisfies
the assumptions in Proposition \ref{prop:individual-harmonicity},
the individual harmonicity is the guiding principle for the first step of our construction.
We define as follows.

\begin{dfn}
	\label{dfn:initial-extension-to-topological-boundary}
	For a $C^\infty$-smooth map $\varphi\colon\cornerX\to\bdry B^{n+1}$
	satisfying \eqref{eq:nondegeneracy},
	we define the map $\tilde{\varphi}\colon\bdryX\to\overline{B^{n+1}}$ by the requirement that
	$\tilde{\varphi}(\orddot,q_2)$ (resp.\ $\tilde{\varphi}(q_1,\orddot)$)
	is the $C^1$ harmonic extension of $\varphi(\orddot,q_2)$ (resp.~$\varphi(q_1,\orddot)$).
\end{dfn}

This definition is justified by classical results of Li--Tam \cite{Li-Tam-93-I,Li-Tam-93-II}.
Any $C^1$ map $\varphi\colon\bdry B^{m+1}\to \bdry B^{n+1}$ with nowhere vanishing differential
admits a $C^1$ extension $u\colon\overline{B^{m+1}}\to\overline{B^{n+1}}$
for which $u(B^{m+1})\subset B^{n+1}$ and $u|_{B^{m+1}}$ is a harmonic map with respect to
the Poincar\'e metrics \cite{Li-Tam-93-II}*{Theorem 6.4}.
The uniqueness of such an extension follows from \cite{Li-Tam-93-I}*{Theorem 1.1}.

Definition \ref{dfn:initial-extension-to-topological-boundary}
leads us to the question on the dependence of the Li--Tam $C^1$ harmonic extension on parameters,
for we want our initial extension $\tilde{\varphi}$ to have some amount of regularity.
Namely, given a $C^\infty$-smooth family of boundary values $\bdry B^{m+1}\to \bdry B^{n+1}$,
we want to show that the resulting family of the Li--Tam $C^1$ extension is also $C^\infty$-smooth in some sense.
We solve this problem when $m=1$ in this section; the restriction on the dimension
is due to technical limitations of our method.
We write $D=B^2$ in what follows.

To proceed further, we recall what have been known on the boundary regularity
of the Li--Tam $C^1$ harmonic extension.
Let $\varphi\colon\bdry B^{m+1}\to\bdry B^{n+1}$ be a $C^\infty$-smooth map
with nowhere vanishing differential.
Then, the harmonic extension $u\in C^1(\overline{B^{m+1}},\overline{B^{n+1}})$
actually belongs to $C^{m+1,\alpha}(\overline{B^{m+1}},\overline{B^{n+1}})$ for any $\alpha\in(0,1)$.
A result close to it is obtained in \cite{Li-Tam-93-II}*{Theorem 5.1}, but precisely,
it is a consequence of the existence of polyhomogeneous expansion of $u$ established by
Economakis \cite{Economakis-93-Thesis}.
Let us specialize to the case $m=1$ for simplicity.
Then, \cite{Economakis-93-Thesis}*{Theorem 1.0.2} implies that
the Li--Tam extension $u\in C^1(\overline{D},\overline{B^{n+1}})$ locally admits,
in the expression $u=(u^0,u^1,\dots,u^n)$ with respect to the upper-half space coordinates,
the asymptotic expansion
\begin{equation}
	\label{eq:polyhomogeneous-expansion}
	u^I(x,y)\sim \varphi^I(y)+x\sum_{l=0}^\infty u^I_l(x,y)(x^2\log x)^l,
	\qquad 0\leqq I\leqq n,
\end{equation}
where each $u^I_l$ is a $C^\infty$-smooth function of the upper-half plane coordinates $(x,y)$.
Moreover, any partial derivative of $u^I$ (of any order) admits a similar
asymptotic expansion obtained by formally differentiating the right-hand side.
It follows from this that $u\in C^{2,\alpha}(\overline{D},\overline{B^{n+1}})$ for any $\alpha\in(0,1)$.
Note also that \eqref{eq:Li-Tam-lemma} implies
\begin{equation}
	\label{eq:Li-Tam-lemma-for-hyperbolic-disk}
	u^0_0(0,y)=\sqrt{e(\varphi)},\qquad
	u^i_0(0,y)=0,\qquad 1\leqq i\leqq n.
\end{equation}

The main result of this section is the following.

\begin{thm}
	\label{thm:smooth-dependence-of-harmonic-extension-on-hyperbolic-space}
	Let $\Phi\colon\bdry D\times U\to\bdry B^{n+1}$ be a $C^\infty$-smooth map,
	where $U\subset\mathbb{R}^l$ is an open set.
	We assume that each $\varphi_t=\Phi(\orddot,t)$ has nowhere vanishing differential,
	and let $u_t\in C^1(\overline{D},\overline{B^{n+1}})$ be the Li--Tam harmonic extension of $\varphi_t$,
	which is automatically $C^{2,\alpha}$ for any $\alpha\in(0,1)$.
	Then, for any $\alpha\in(0,1)$, $u_t$ is a $C^\infty$-smooth family of
	maps in $C^{2,\alpha}(\overline{D},\overline{B^{n+1}})$.
\end{thm}

To prove this,
we use a technique similar to that of Graham--Lee \cite{Graham-Lee-91}
for deforming conformally compact Einstein metrics,
which is extended later by Biquard \cite{Biquard-00} and Lee \cite{Lee-06}.

Without loss of generality, we may assume $0\in U$ and it suffices to show that $u_t$ is smooth in $t$ near $0$.
We consider the vector bundle $E=u_0^*TB^{n+1}$ over $D$;
note that $E$ is equipped with the induced connection.
The tension field operator is reinterpreted below as an operator from the space of sections of $E$ to itself.

The space $L^2(E)$ of square-integrable sections and the $L^2$-Sobolev spaces $H^k(E)$ are defined as usual.
For any $k$ times continuously differentiable section $\sigma$ of $E$, we set
\begin{equation}
	\norm{\sigma}_{C^k}=\sum_{j=0}^k\sup\abs{\nabla^j\sigma}
\end{equation}
and, for any $\alpha\in(0,1)$,
\begin{equation}
	\norm{\sigma}_{C^{k,\alpha}}=\norm{\sigma}_{C^k}+
	\sup_{\substack{p,\,q\in D\\ p\not=q}}
	\frac{\abs{\Pi_{q\to p}(\nabla^k\sigma)(q)-(\nabla^k\sigma)(p)}}{d(p,q)^\alpha},
\end{equation}
where $d(p,q)$ is the hyperbolic distance and
$\Pi_{q\to p}$ denotes the parallel transport from $q$ to $p$
along the unique geodesic in $D$ passing through the two points.
The space $C^k(E)$ (resp.~$C^{k,\alpha}(E)$) consists of sections $\sigma$ for which
$\norm{\sigma}_{C^k}$ (resp.~$\norm{\sigma}_{C^{k,\alpha}}$) is finite.

\begin{rem}
	\label{rem:equivalent-Holder-norms}
	(1) Obviously, for any fixed $r>0$,
	\begin{equation}
		\norm{\sigma}_{C^k}+
		\sup_{\substack{p,\,q\in D\\ p\not=q,\,d(p,q)<r}}
		\frac{\abs{\Pi_{q\to p}(\nabla^k\sigma)(q)-(\nabla^k\sigma)(p)}}{d(p,q)^\alpha}
	\end{equation}
	is a norm equivalent to $\norm{\sigma}_{C^{k,\alpha}}$.

	(2) Another equivalent norm to $\norm{\orddot}_{C^{k,\alpha}}$ can be defined
	through local normal coordinates as follows.
	It is an immediate consequence from \eqref{eq:Li-Tam-lemma} or \eqref{eq:polyhomogeneous-expansion}
	that $\abs{du_0}$ is bounded.
	Therefore, for an arbitrarily fixed $R>0$, we can take an $r>0$ so that
	the image of any geodesic disk $D(p;r)\subset D$ of radius $r$
	is contained in the geodesic ball $B(u_0(p);R)\subset B^{n+1}$.
	We fix such an $r$, and
	let $\set{D_\lambda}$ be a uniformly locally finite covering of $D$ consisting of
	geodesic disks $D_\lambda=D(p_\lambda;r)$ such that any $p$, $q\in D$ with $d(p,q)<r/2$ is
	contained in some $D_\lambda$
	(such a covering can be found by letting $\set{p_\lambda}$ be a maximal set of points for which
	any two $D(p_\lambda;r/4)$ are disjoint; see Lee \cite{Lee-06}*{Lemma 2.2}).
	Then the existence of polyhomogeneous expansion of $u_0$ implies that, if $u_0$ is
	expressed locally as $u_0=(u_0^1,u_0^2,\dots,u_0^{n+1})$
	in the normal coordinates in $D(p_\lambda;r)$ and $B(u_0(p_\lambda);R)$,
	the functions $u_0^i$, $1\leqq i\leqq n+1$, are $C^\infty$-bounded uniformly in $\lambda$.
	Therefore,
	the Christoffel symbols of the induced connection on $E$ in these normal coordinates
	are also $C^\infty$-bounded uniformly in $\lambda$.
	As a consequence,
	if $s_\lambda\colon D_r\to\mathbb{R}^{n+1}$ denotes the local expression of $\sigma$
	in terms of the normal coordinates in $D(p_\lambda;r)$ and $B(u_0(p);R)$
	(where $D_r\subset\mathbb{R}^2$ is the disk of radius $r$ centered at the origin), then
	$\sup_\lambda\norm{s_\lambda}_{C^{k,\alpha}}$
	gives another equivalent norm to $\norm{\orddot}_{C^{k,\alpha}}$.
\end{rem}

Recall that the tension field of $u\colon D\to B^{n+1}$ is given by
\begin{equation}
	\tau(u)^k=
	\tensor{g}{^a^b}(\partial_a\partial_bu^k-\tensor{\Gamma}{^c_a_b}\partial_cu^k
	+(\tensor{\Gamma}{^k_i_j}\compose u)(\partial_au^i)(\partial_bu^j))
\end{equation}
in terms local coordinates,
where $g$ denotes the Poincar\'e metric on $D$.

A one-to-one correspondence between mappings $u\colon D\to B^{n+1}$ and
sections $\sigma\in\Gamma(E)$ can be defined through the exponential map:
\begin{equation}
	u(p)=\exp_{u_0(p)}\sigma(p).
\end{equation}
Since $u_0$ is $C^\infty$-smooth in $D$,
$u$ is locally $C^{k,\alpha}$ if and only if $\sigma$ is locally $C^{k,\alpha}$.
This correspondence is used to regard the tension field operator $\tau$
as acting on sections $\sigma$ rather than on $u$.
Moreover, by using the parallel translation from $u(p)$ to $u_0(p)$ along the geodesic in $B^{n+1}$
to interpret a section of $u^*TB^{n+1}$ as a section of $E=u_0^*TB^{n+1}$, we obtain
a smooth nonlinear mapping
\begin{equation}
	T\colon C^{k+2,\alpha}(E)\to C^{k,\alpha}(E);
\end{equation}
the fact that $\sigma\in C^{k+2,\alpha}(E)$ implies $T(\sigma)\in C^{k,\alpha}(E)$
can be easily seen by the equivalent norm
discussed in Remark \ref{rem:equivalent-Holder-norms} (2).

Let $\rho\in C^\infty(\overline{D})$ denote any boundary defining function.
We introduce the weighted H\"older spaces $C^{k,\alpha}_\delta(E)$ for any $\delta\in\mathbb{R}$ by
\begin{equation}
	C^{k,\alpha}_\delta(E)=\rho^\delta C^{k,\alpha}(E)
\end{equation}
with norm given by
\begin{equation}
	\norm{\sigma}_{C^{k,\alpha}_\delta}=\norm{\rho^{-\delta}\sigma}_{C^{k,\alpha}}.
\end{equation}
Then, noting the fact that $\partial_a\rho/\rho$ and its consecutive partial derivatives
in the normal coordinates are all bounded, we can deduce that
the tension field operator further gives a smooth mapping
\begin{equation}
	\label{eq:harmonic-map-Laplacian-on-weighted-Holder-spaces}
	T\colon C^{k+2,\alpha}_\delta(E)\to C^{k,\alpha}_\delta(E)
\end{equation}
for any $\delta\geqq 0$.

The linearization $P$ of $T$ at $0$ is given by
\begin{equation}
	\label{eq:local-formula-of-linearization-for-single-hyperbolic-plane}
	(P\sigma)^k
	=\tensor{g}{^a^b}
	(\partial_a\partial_b\sigma^k-\tensor{\Gamma}{^c_a_b}\partial_c\sigma^k
	+((\partial_l\tensor{\Gamma}{^k_i_j})\compose u_0)\sigma^l(\partial_au_0^i)(\partial_bu_0^j)
	+2(\tensor{\Gamma}{^k_i_j}\compose u_0)(\partial_au_0^i)(\partial_b\sigma^j)),
\end{equation}
which can also be expressed as
\begin{equation}
	\label{eq:Weizenbock-formula-for-linearization-for-single-hyperbolic-plane}
	(P\sigma)^k=-\nabla^*\nabla\sigma^k
	+\tensor{g}{^a^b}\tensor{R}{_i_j^k_l}\sigma^i(\partial_au_0^j)(\partial_bu_0^l).
\end{equation}
This gives a bounded linear operator
\begin{equation}
	\label{eq:linearization-for-single-hyperbolic-plane-on-weighted-Holder-spaces}
	P\colon C^{k+2,\alpha}_\delta(E)\to C^{k,\alpha}_\delta(E)
\end{equation}
for any $k\in\mathbb{Z}_{\geqq 0}$, $\alpha\in(0,1)$, $\delta\in\mathbb{R}$,
and a bounded linear operator
\begin{equation}
	\label{eq:linearization-for-single-hyperbolic-plane-on-Sobolev-spaces}
	P\colon H^2(E)\to L^2(E).
\end{equation}

If $du_0$ has rank greater than $1$ everywhere,
\eqref{eq:Weizenbock-formula-for-linearization-for-single-hyperbolic-plane} immediately implies that
there exists a constant $C>0$ such that
\begin{equation}
	\label{eq:fundamental-estimate}
	\norm{\sigma}_{L^2}\leqq C\norm{P\sigma}_{L^2}
\end{equation}
for any compactly supported section $\sigma\in C^\infty_c(E)$,
and hence for any $\sigma\in H^2(E)$ by continuity.
Consequently, \eqref{eq:linearization-for-single-hyperbolic-plane-on-Sobolev-spaces} has trivial kernel.
The same conclusion can be drawn as follows even if $du_0$ does not satisfy the above assumption
(cf.\ \cite{Lee-06}*{p.\ 78}).
If we suppose that $\sigma\in H^2(E)$ satisfies $P\sigma=0$,
then $\tensor{g}{^a^b}\tensor{R}{_i_j_k_l}\sigma^i\sigma^k(\partial_au_0^j)(\partial_bu_0^l)$
must vanish identically on $D$.
Choose a compact subset $K\subset D$ so that $\rank(du_0)_p>1$ for any $p\in D\setminus K$
(which is possible by \eqref{eq:Li-Tam-lemma}).
Then, $\sigma$ vanishes in $D\setminus K$,
and the weak unique continuation property \cite{Aronszajn-57}*{Remark 3} implies that $\sigma$ is
identically zero in $D$.

Additionally, no matter whether $\rank(du_0)_p>1$  everywhere or not,
the fact that \eqref{eq:fundamental-estimate} is satisfied by any $\sigma\in C^\infty_c(E)$ supported outside
some fixed compact subset $K\subset D$
implies that \eqref{eq:linearization-for-single-hyperbolic-plane-on-Sobolev-spaces} has
closed range (cf.\ \cite{Lee-06}*{Lemma 4.10}).
Moreover, if $\sigma\in(\ran P)^\perp$,
then $P\sigma=0$ in the sense of distribution since $P$ is formally self-adjoint.
The global elliptic regularity for $P$ (cf.\ \cite{Lee-06}*{Lemma 4.8}) implies that $\sigma\in H^2(E)$,
and hence $\sigma=0$.
This proves that \eqref{eq:linearization-for-single-hyperbolic-plane-on-Sobolev-spaces} is an isomorphism.

We claim the following.

\begin{prop}
	\label{prop:isomorphism-theorem-for-linearized-harmonic-map-Laplacian}
	For $-1<\delta<2$,
	\eqref{eq:linearization-for-single-hyperbolic-plane-on-weighted-Holder-spaces} is an isomorphism.
\end{prop}

\begin{proof}
	We argue along the lines of \cite{Lee-06}*{Chapters 4--6}.
	Although $P$ is not a geometric elliptic differential operator in the sense of \cite{Lee-06},
	the same method can be applied, thanks to the fact that
	$u_0$ can be locally approximated well by a certain standard mapping that we write $\breve{u}$ below.
	This is a special feature of the case $m=1$.

	Let $\breve{u}\colon D\to B^{n+1}$ denote the totally geodesic isometric embedding of the hyperbolic plane
	that, in terms of the upper-half space coordinates, is given by
	\begin{equation}
		\breve{u}(x,y)=(x,y,0,\dots,0).
	\end{equation}
	Let $\breve{E}=\breve{u}^*TB^{n+1}$ and $\breve{P}$ the linearized tension field operator
	\eqref{eq:Weizenbock-formula-for-linearization-for-single-hyperbolic-plane} associated with $\breve{u}$.
	Then, $\breve{P}\colon H^2(\breve{E})\to L^2(\breve{E})$ is an isomorphism as we have already seen.
	We introduce the Green's kernel of $\breve{P}$, which we write $G(p_1,p_2)$, $p_1$, $p_2\in D$,
	as the Schwartz kernel for $\breve{P}^{-1}$.
	It is $C^\infty$-smooth away from the diagonal,
	and if we write $G(p)=G(0,p)\in\Hom(\breve{E}_0,\breve{E}_p)$, where $0\in D$ is the center,
	then its asymptotic behavior as $\rho(p)\to 0$ can be
	determined as in \cite{Lee-06}*{Proposition 5.2} based on the following explicit formula of $\breve{P}$:
	in the upper-half space coordinates, we have
	\begin{align}
		(\breve{P}\sigma)^0&=-(x\partial_x)(x\partial_x-3)\sigma^0-x^2\partial_y^2\sigma^0-2x\partial_y\sigma^1,\\
		(\breve{P}\sigma)^1&=-(x\partial_x)(x\partial_x-3)\sigma^1-x^2\partial_y^2\sigma^1+2x\partial_y\sigma^0,\\
		(\breve{P}\sigma)^i&=-(x\partial_x)(x\partial_x-3)\sigma^i-x^2\partial_y^2\sigma^i,\qquad 2\leqq i\leqq n.
	\end{align}
	From this, it follows that $G(p)$ satisfies, for any $\varepsilon>0$,
	\begin{equation}
		\abs{G(p)}\leqq C_\varepsilon\rho(p)^{2-\varepsilon}.
	\end{equation}
	It is a consequence from this estimate that $\breve{P}$ induces an isomorphism
	$\breve{P}\colon C^{k+2,\alpha}_\delta(\breve{E})\to C^{k,\alpha}_\delta(\breve{E})$
	for any $k$, $\alpha\in (0,1)$, and $-1<\delta<2$ (cf.~\cite{Lee-06}*{Theorem 5.9}).

	From now, we turn to the operator $P$ for general $u_0$.
	In view of \eqref{eq:polyhomogeneous-expansion} and \eqref{eq:Li-Tam-lemma-for-hyperbolic-disk},
	for each point $q\in\bdry D$,
	we can take an upper-half plane coordinates $(V;x,y)$ centered around $q\in\bdry D$ and
	an upper-half space coordinates $(U;\xi,\eta^1,\eta^2,\dots,\eta^n)$ centered around $u_0(q)\in\bdry B^{n+1}$
	for which $u_0(V)\subset U$ and
	\begin{align}
		u_0^0(x,y)&=x\psi^0(y)+\dotsb,\\
		u_0^1(x,y)&=\varphi^1(y)+x\psi^1(y)+\dotsb,\\
		u_0^i(x,y)&=\varphi^i(y)+x\psi^i(y)+\dotsb,\qquad 2\leqq i\leqq n,
	\end{align}
	with
	\begin{equation}
		\varphi^1(0)=1,\qquad
		\varphi^2(0)=\dots=\varphi^n(0)=0,
	\end{equation}
	and
	\begin{equation}
		\psi^0(0)=1,\qquad
		\psi^1_1(0)=\psi^2_1(0)=\dots=\psi^n_1(0)=0.
	\end{equation}
	Using these coordinates, we may interpret
	any section $\sigma\in C^{k,\alpha}_\delta(E)$ supported in $V$ as a section in $C^{k,\alpha}_\delta(\breve{E})$.
	By taking an appropriate covering of a neighborhood of $\bdry D$ that consists of such $V$,
	we can construct, following the method in \cite{Lee-06}*{Chapter 6},
	a parametrix $Q\colon L^2(E)\to H^2(E)$ to the operator $P\colon H^2(E)\to L^2(E)$.
	This $Q$ naturally induces a bounded operator
	\begin{equation}
		Q\colon C^{k,\alpha}_\delta(E)\to C^{k+2,\alpha}_\delta(E),\qquad -1<\delta<2,
	\end{equation}
	which turns out to be a parametrix to $P\colon C^{k+2,\alpha}_\delta(E)\to C^{k,\alpha}_\delta(E)$.
	Therefore, $P\colon C^{k+2,\alpha}_\delta(E)\to C^{k,\alpha}_\delta(E)$ is a Fredholm operator,
	and it can be checked that its index is zero.
	Furthermore, the careful construction of the parametrix $Q$ shows that
	the kernel of $P\colon C^{k+2,\alpha}_\delta(E)\to C^{k,\alpha}_\delta(E)$
	can be identified with that of $P\colon H^2(E)\to L^2(E)$.
	However, we know that $P\colon H^2(E)\to L^2(E)$ has trivial kernel.
	This proves that so is
	\eqref{eq:linearization-for-single-hyperbolic-plane-on-weighted-Holder-spaces}.
\end{proof}

To show Theorem \ref{thm:smooth-dependence-of-harmonic-extension-on-hyperbolic-space},
by following the procedure of \cite{Li-Tam-93-I}*{Lemma 2.2},
we take an extension $\tilde{u}_t\in C^\infty(\overline{D},\overline{B^{n+1}})$ of $\varphi_t$
for each $t$, $C^\infty$-smoothly in $t$, as a preliminary extension in such a way that
$\tilde{u}_0=u_0$ and
\begin{equation}
	\abs{\tau(\tilde{u}_t)}=O(\rho^{1+\alpha})
\end{equation}
is satisfied.
For a sufficiently small $\sigma\in C^{k+2,\alpha}(E)$,
we express as
\begin{equation}
	\tilde{u}_t=(\tilde{u}^1_t,\tilde{u}^2_t,\dots,\tilde{u}^{n+1}_t),\qquad
	\sigma=(\sigma^1,\sigma^2,\dots,\sigma^{n+1})
\end{equation}
using the Euclidean coordinates in $B^{n+1}$
to define the mapping $\tilde{u}_t+\sigma\colon D\to B^{n+1}$ by
\begin{equation}
	(\tilde{u}_t+\sigma)^i=\tilde{u}^i_t+\sigma^i.
\end{equation}
Using this, we consider the mapping
\begin{equation}
	\mathcal{Q}\colon U\times C^{2,\alpha}_{1+\alpha}(E)\to C^{0,\alpha}_{1+\alpha}(E),
	\qquad
	(t,\sigma)\mapsto T(\tilde{u}_t+\sigma).
\end{equation}
Then, Proposition \ref{prop:isomorphism-theorem-for-linearized-harmonic-map-Laplacian} and
the implicit function theorem shows that, by shrinking $U$ if necessary,
we can find $\sigma_t\in C^{2,\alpha}_{1+\alpha}(E)$ for each $t\in U$,
$C^\infty$-smoothly in $t$, such that $\mathcal{Q}(\tilde{u}_t,\sigma_t)=T(\tilde{u}_t+\sigma_t)=0$
and hence $\tau(\tilde{u}_t+\sigma_t)=0$.
An argument similar to \cite{Lee-06}*{Lemma 3.7} shows that
each $\tilde{u}_t+\sigma_t$ lies in $C^{2,\alpha}(\overline{D},\overline{B^{n+1}})$
and is $C^\infty$-smooth in $t$ as a family of elements of $C^{2,\alpha}(\overline{D},\overline{B^{n+1}})$.
This completes the proof of Theorem \ref{thm:smooth-dependence-of-harmonic-extension-on-hyperbolic-space}.

For later use, we further remark that, if $0\leqq\beta\leqq\alpha$, then
\begin{equation}
	\rho^{-2-\beta}\nabla^k\sigma_t\in C^{2-k,\alpha}_{-1+\alpha-\beta}((T^*D)^{\otimes k}\otimes E)
\end{equation}
for $k=0$, $1$, $2$,
which implies that
\begin{equation}
	\rho^{-2-\beta}\sigma_t^i,\qquad
	\rho^{-1-\beta}\partial_a\sigma_t^i,\qquad
	\rho^{-\beta}\partial_a\partial_b\sigma_t^i
\end{equation}
are all in $C^{\alpha-\beta}(\overline{D})$, where the partial derivatives are taken with respect to
the Euclidean coordinates in $D$.

\section{Construction of an approximate solution}
\label{sec:approx-solution}

By Theorem \ref{thm:smooth-dependence-of-harmonic-extension-on-hyperbolic-space},
for the initial extension $\tilde{\varphi}\colon\bdry X\to\overline{B^{n+1}}$
of $\varphi\colon\cornerX\to\bdry B^{n+1}$
defined in Definition \ref{dfn:initial-extension-to-topological-boundary},
$\tilde{\varphi}(\orddot,q_2)$
(resp.~$\tilde{\varphi}(q_1,\orddot)$) is $C^\infty$-smooth in $q_2$ (resp.~in $q_1$).
We construct an approximate harmonic map $u\colon X=D\times D\to B^{n+1}$
by extending $\tilde{\varphi}$ in this section.

In this dimension, the tension field formula \eqref{eq:tension-field} reads
\begin{subequations}
	\label{eq:tension-field-for-bidisk}
	\begin{align}
		\label{eq:tension-field-for-bidisk-normal}
		\frac{1}{u^0}\tau^0
		&=\sum_{\lambda=1}^{2}
		x_\lambda^2
		\left(\frac{1}{u^0}\Delta_{\overline{g}_\lambda}u^0
		-\frac{1}{(u^0)^2}\left(\abs{\nabla_{\overline{g}_\lambda}u^0}_{\overline{g}_\lambda}^2
		-\sum_{i=1}^n\abs{\nabla_{\overline{g}_\lambda}u^i}_{\overline{g}_\lambda}^2\right)\right),\\
		\label{eq:tension-field-for-bidisk-tangential}
		\frac{1}{u^0}\tau^i
		&=\sum_{\lambda=1}^{2}
		x_\lambda^2
		\left(\frac{1}{u^0}\Delta_{\overline{g}_\lambda}u^i
		-\frac{1}{(u^0)^2}
		\braket{\nabla_{\overline{g}_\lambda}u^0,\nabla_{\overline{g}_\lambda}u^i}_{\overline{g}_\lambda}\right),
		\qquad 1\leqq i\leqq n.
	\end{align}
\end{subequations}

The next proposition illustrates the main idea for our construction of an approximate solution
for the harmonic map equation.

Let $\rho$ be any fixed boundary defining function of $D$. We set
\begin{equation}
	Q_{\rho_0}=\set{0<\rho<\rho_0}\times\set{0<\rho<\rho_0}\subset X
\end{equation}
and
\begin{equation}
	\overline{Q_{\rho_0}}=\set{0\leqq\rho<\rho_0}\times\set{0\leqq\rho<\rho_0}\subset\overline{X}.
\end{equation}
Moreover, we write $\rho_\lambda=\rho\compose\pi_\lambda$, $\lambda=1$, $2$,
where $\pi_\lambda\colon X\to D$ is the projection to the $\lambda$-th factor.

\begin{prop}
	\label{prop:sufficient-conditions-of-approximate-solution}
	Let $\alpha\in(0,1)$.
	Suppose that $u\in C^{2,\alpha}(\overline{X},\overline{B^{n+1}})$ satisfies
	\eqref{eq:properness-at-corner}, \eqref{eq:nondegeneracy}, and $u|_{\bdry X}=\tilde{\varphi}$.
	Then, the pointwise norm of its tension field satisfies
	$\abs{\tau(u)}=O(\rho_1^{\alpha/2}\rho_2^{\alpha/2})$ near $\bdry X$.
\end{prop}

\begin{proof}
	It suffices to prove the claim in $Q_{\rho_0}$ for some $\rho_0>0$.
	In fact, note that the image of $X\setminus Q_{\rho_0}$ is uniformly away from $\bdry B^{n+1}$
	by the condition \eqref{eq:properness-at-corner}.
	For any fixed $(q_1,p_2)\in \bdry D\times\set{\rho\geqq\rho_0}$,
	introduce upper-half plane coordinates around $q_1$ and around $p_2$
	to use formula \eqref{eq:tension-field-for-bidisk} for the tension field.
	Then, since $u(q_1,\orddot)=\tilde{\varphi}(q_1,\orddot)\colon D\to B^{n+1}$ is a harmonic map,
	the $\lambda=2$ term in the right-hand side of each of
	\eqref{eq:tension-field-for-bidisk-normal} and \eqref{eq:tension-field-for-bidisk-tangential}
	vanishes when $x_1=0$, which implies that this term is $O(x_1^\alpha)$ in a neighborhood of $(q_1,p_2)$.
	On the other hand, the $\lambda=1$ terms in
	\eqref{eq:tension-field-for-bidisk-normal} and \eqref{eq:tension-field-for-bidisk-tangential}
	are clearly $O(x_1^2)$.
	Therefore, we have $\abs{\tau(u)}=O(\rho_1^\alpha)=O(\rho_1^\alpha\rho_2^\alpha)$
	in a neighborhood of $(q_1,p_2)$.
	As the same argument also works in a neighborhood of $(p_1,q_2)\in \set{\rho\geqq\rho_0}\times\bdry D$,
	we may conclude that $\abs{\tau(u)}=O(\rho_1^\alpha\rho_2^\alpha)$ in $X\setminus Q_{\rho_0}$.

	To show that $\abs{\tau(u)}=O(\rho_1^{\alpha/2}\rho_2^{\alpha/2})$ in $Q_{\rho_0}$,
	it suffices to take any $(q_1,q_2)\in\cornerX$ and argue locally around this point
	by using the upper-half plane coordinates $(x_1,y_1,x_2,y_2)$ defined near $(q_1,q_2)$.

	Let us write $u=(u^0,u^1,\dots,u^n)$ in the upper-half space coordinates.
	Then, as in the proof of Proposition \ref{prop:dimension-restriction},
	we have
	\begin{equation}
		u^0=a_1x_1+a_2x_2+O(x_1^2,x_1x_2,x_2^2)
	\end{equation}
	in a neighborhood of $(q_1,q_2)$,
	where $a_\lambda=a_\lambda(y_1,y_2)=\sqrt{e_\lambda(\varphi)}$, $\lambda=1$, $2$.

	Since $u(\orddot,q'_2)\colon D\to B^{n+1}$ is harmonic for each $q'_2\in\bdry D$,
	\begin{equation}
		\label{eq:conclusion-from-boundary-harmonicity}
		u^0\Delta_{\overline{g}_1}u^0
		-\left(\abs{\nabla_{\overline{g}_1}u^0}^2-\sum_{i=1}^n\abs{\nabla_{\overline{g}_1}u^i}^2\right)
	\end{equation}
	vanishes at any point of the form $(x_1,y_1,0,y_2)$.
	Therefore, we can deduce that
	\begin{equation}
		\label{eq:part-of-tension-field-decay-1}
		u^0\Delta_{\overline{g}_1}u^0
		-\left(\abs{\nabla_{\overline{g}_1}u^0}^2-\sum_{i=1}^n\abs{\nabla_{\overline{g}_1}u^i}^2\right)
		=O(x_2^\alpha)
	\end{equation}
	in a neighborhood of $(q_1,q_2)$. Similarly, we obtain
	\begin{equation}
		\label{eq:part-of-tension-field-decay-2}
		u^0\Delta_{\overline{g}_2}u^0
		-\left(\abs{\nabla_{\overline{g}_2}u^0}^2-\sum_{i=1}^n\abs{\nabla_{\overline{g}_2}u^i}^2\right)
		=O(x_1^\alpha).
	\end{equation}
	Note moreover that a neighborhood of $(q_1,q_2)$ can be taken so small that
	\begin{equation}
		\label{eq:estimate-from-below-near-corner}
		u^0>\frac{a_1}{2}x_1+\frac{a_2}{2}x_2
	\end{equation}
	is satisfied there.
	In a neighborhood taken this way, by \eqref{eq:part-of-tension-field-decay-1}
	and \eqref{eq:part-of-tension-field-decay-2},
	\begin{equation*}
		\begin{split}
			\frac{1}{u^0}\tau^0
			&=\frac{1}{(u^0)^2}x_1^2\cdot O(x_2^\alpha)+\frac{1}{(u^0)^2}x_2^2\cdot O(x_1^\alpha)\\
			&=\left(\frac{x_1}{u^0}\right)^{2-\alpha/2}\left(\frac{x_2}{u^0}\right)^{\alpha/2}
			\cdot O(x_1^{\alpha/2}x_2^{\alpha/2})
			+\left(\frac{x_2}{u^0}\right)^{2-\alpha/2}\left(\frac{x_1}{u^0}\right)^{\alpha/2}
			\cdot O(x_1^{\alpha/2}x_2^{\alpha/2}).
		\end{split}
	\end{equation*}
	By \eqref{eq:estimate-from-below-near-corner},
	we have $x_1/u^0<2a_1^{-1}$ and $x_2/u^0<2a_2^{-1}$,
	from which it follows that $\tau^0/u^0=O(x_1^{\alpha/2}x_2^{\alpha/2})$.
	In the same manner, $\tau^i/u^0=O(x_1^{\alpha/2}x_2^{\alpha/2})$ can be shown.
	Therefore, the estimate $\abs{\tau(u)}=O(x_1^{\alpha/2}x_2^{\alpha/2})$ holds in a neighborhood $(q_1,q_2)$.
\end{proof}

Note that the computation in the above proof fails unless $m_1=m_2=1$.
For example, if $m_1\not=1$, \eqref{eq:conclusion-from-boundary-harmonicity} contains an additional term
$(m_1-1)x_1^{-1}u^0\partial_{x_1}u^0$, which is not continuous up to the boundary of $\overline{X}$ and
destroys the subsequent argument.

\begin{prop}
	\label{prop:approx-solution}
	For any $C^\infty$ map $\varphi\colon\cornerX\to\bdry B^{n+1}$
	satisfying \eqref{eq:nondegeneracy},
	we may construct a map $u\colon\overline{X}\to\overline{B^{n+1}}$
	satisfying \eqref{eq:properness-at-corner} and $u|_\bdryX=\tilde{\varphi}$
	so that $u\in C^{2,\alpha}(\overline{X},\overline{B^{n+1}})$ for any $\alpha\in(0,1)$.
	By Proposition \ref{prop:sufficient-conditions-of-approximate-solution},
	it satisfies
	the estimate $\abs{\tau(u)}=O(\rho_1^{\alpha/2}\rho_2^{\alpha/2})$ for any $\alpha\in(0,1)$ near $\bdryX$.
\end{prop}

\begin{proof}
	The claim follows if such an extension of $\tilde{\varphi}$ exists in a neighborhood of each point on $\bdryX$.
	In fact, if we can cover a neighborhood of $\bdry X$ with open sets of $\overline{X}$ in which
	$\tilde{\varphi}$ admits such an extension, then we may regard those extensions as $\mathbb{R}^{n+1}$-valued
	functions and patch them up using a partition of unity, using the convexity of
	the ball in $\mathbb{R}^{n+1}$.
	The map from a neighborhood of $\bdry X$ into $\overline{B^{n+1}}$ constructed this way
	can easily be extended to a map defined on the whole $\overline{X}$ without changing its values near $\bdryX$.

	Furthermore, it is clear that a desired extension can be locally constructed near any point on
	$\bdryX\setminus\cornerX$.
	Therefore, it suffices if we can construct an extension $u$ of $\tilde{\varphi}$
	around an arbitrarily fixed point $(q_1,q_2)\in\bdry D\times\bdry D$ so that
	$u$ is $C^{2,\alpha}$ for any $\alpha\in(0,1)$.
	Take an open neighborhood $\overline{U}_\lambda$ of each $q_\lambda$ in $\overline{D}$ so that
	$\overline{U}_\lambda$ is a rectangle sitting in an upper-half plane model.
	Let $\bdry U_\lambda=\overline{U}_\lambda\cap\set{x_\lambda=0}$ and
	$U_\lambda=\overline{U}_\lambda\setminus\bdry U_\lambda$.

	We may assume that $\tilde{\varphi}(\bdryX\cap(\overline{U}_1\times\overline{U}_2))$
	is contained in the closed upper-half space,
	and let us write $\tilde{\varphi}=(\tilde{\varphi}^0,\tilde{\varphi}^1,\dots,\tilde{\varphi}^n)$
	using such upper-half space coordinates.
	Then, the proof of Theorem \ref{thm:smooth-dependence-of-harmonic-extension-on-hyperbolic-space} shows that
	\begin{align}
		\tilde{\varphi}^0(x_1,y_1,0,y_2)
		&=\sum_{k=1}^2 x_1^k\psi_{1,k}^0(y_1,y_2)+R_1^0(x_1,y_1,y_2),\\
		\tilde{\varphi}^i(x_1,y_1,0,y_2)
		&=\varphi^i(y_1,y_2)+\sum_{k=1}^2 x_1^k\psi_{1,k}^i(y_1,y_2)+R_1^i(x_1,y_1,y_2), \qquad 1\leqq i\leqq n,\\
		\tilde{\varphi}^0(0,y_1,x_2,y_2)
		&=\sum_{k=1}^2 x_2^k\psi_{2,k}^0(y_1,y_2)+R_2^0(x_2,y_1,y_2),\\
		\tilde{\varphi}^i(0,y_1,x_2,y_2)
		&=\varphi^i(y_1,y_2)+\sum_{k=1}^2 x_2^k\psi_{2,k}^i(y_1,y_2)+R_2^i(x_2,y_1,y_2), \qquad 1\leqq i\leqq n,
	\end{align}
	where $\psi_{\lambda,k}^i$, $0\leqq i\leqq n$, are $C^\infty$-smooth and
	$R_\lambda^i$, $0\leqq i\leqq n$, are $C^{2,\alpha}$ for any $\alpha\in(0,1)$.
	Moreover, the remark at the end of the last section implies
	\begin{equation}
		\abs{\nabla_{\overline{g}_\lambda}^kR_\lambda^i(x_\lambda,y_1,y_2)}=O(x_\lambda^{2-k+\alpha})
	\end{equation}
	and
	\begin{equation}
		x_\lambda^{-2+k}\abs{\nabla_{\overline{g}_\lambda}^kR_\lambda^i(x_\lambda,y_1,y_2)}\in
		C^\alpha(\overline{U}_1\times\overline{U}_2)
	\end{equation}
	for $k=0$, $1$, $2$.

	We construct a map $u$ defined in $\overline{U}_1\times \overline{U}_2$
	that extends $\tilde{\varphi}$ in the following way.
	We set
	\begin{equation}
		u^0(x_1,y_1,x_2,y_2)
		=\sum_{k=1}^2 x_1^k\psi_{1,k}^0(y_1,y_2)+\sum_{k=1}^2 x_2^k\psi_{2,k}^0(y_1,y_2)+R^0(x_1,y_1,x_2,y_2)
	\end{equation}
	and
	\begin{equation}
		u^i(x_1,y_1,x_2,y_2)
		=\varphi^i(y_1,y_2)
		+\sum_{k=1}^2 x_1^k\psi_{1,k}^i(y_1,y_2)
		+\sum_{k=1}^2 x_2^k\psi_{2,k}^i(y_1,y_2)+R^i(x_1,y_1,x_2,y_2),
	\end{equation}
	where we define $R^i$, $0\leqq i\leqq n$, as follows---we use real blowup around the corner.
	Note that
	\begin{equation*}
		\theta(x_1,x_2)=\arctan\left(\frac{x_2}{x_1}\right)
	\end{equation*}
	is a smooth function on $(\overline{U}_1\times\overline{U}_2)\setminus\set{x_1=x_2=0}$
	taking values in $[0,\pi/2]$.
	Let $\chi_1$ and $\chi_2$ be smooth functions $[0,\pi/2]\to\mathbb{R}_{\geqq 0}$
	such that $\chi_1\equiv 1$ near $0$, $\chi_2\equiv 1$ near $\pi/2$, and $\chi_1+\chi_2\equiv 1$.
	We define, in $(\overline{U}_1\times\overline{U}_2)\setminus\set{x_1=x_2=0}$,
	\begin{equation}
		\label{eq:approximate-solution-merge-R}
		R^i(x_1,y_1,x_2,y_2)=\chi_1(\theta(x_1,x_2))R^i_1(x_1,y_1,y_2)+\chi_2(\theta(x_1,x_2))R^i_2(x_2,y_1,y_2).
	\end{equation}
	On $\set{x_1=x_2=0}$, we set $R^i(0,y_1,0,y_2)=0$.
	Then the map $u\colon\overline{U}_1\times\overline{U}_2\to\overline{H_+^{n+1}}$
	defined by $(u^0,u^1,\dots,u^n)$ is a continuous extension of $\tilde{\varphi}$,
	and the image of $(\overline{U}_1\times \overline{U}_2)\setminus\set{x_1=x_2=0}$ is actually
	contained in $H_+^{n+1}$ if $\overline{U}_1$ and $\overline{U}_2$ are sufficiently small
	because $\psi^0_{\lambda,1}=\sqrt{e_\lambda(\varphi)}>0$.

	We shall verify that $u=(u^0,u^1,\dots,u^n)$ so constructed is
	$C^{2,\alpha}$ in $\overline{U}_1\times \overline{U}_2$.
	It suffices to show that the functions $R^i(x_1,y_1,x_2,y_2)$ are $C^{2,\alpha}$.
	We are going to prove that each term on the right-hand side of
	\eqref{eq:approximate-solution-merge-R} is $C^{2,\alpha}$.

	We only show that $\chi_1(\theta(x_1,x_2))R^i_1(x_1,y_1,y_2)$ is $C^{2,\alpha}$; the other term can be
	treated similarly. Note that
	\begin{alignat*}{3}
		\partial_{x_1}\theta&=\frac{-x_2}{x_1^2+x_2^2}, \qquad
		&\partial_{x_1}^2\theta&=\frac{2x_1x_2}{(x_1^2+x_2^2)^2},\\
		\partial_{x_2}\theta&=\frac{x_1}{x_1^2+x_2^2}, \qquad
		&\partial_{x_2}^2\theta&=\frac{-2x_1x_2}{(x_1^2+x_2^2)^2}, \qquad
		&\partial_{x_1}\partial_{x_2}\theta&=\frac{-x_1^2+x_2^2}{(x_1^2+x_2^2)^2}.
	\end{alignat*}
	The first derivatives of $\chi_1(\theta(x_1,x_2))R_1^i(x_1,y_1,y_2)$ are therefore
	\begin{equation*}
		\partial_{x_1}(\chi_1(\theta)R_1^i)
		=\frac{-x_2}{x_1^2+x_2^2}\chi_1'(\theta)R_1^i+\chi_1(\theta)\partial_{x_1}R_1^i
	\end{equation*}
	and
	\begin{equation*}
		\partial_{x_2}(\chi_1(\theta)R_1^i)
		=\frac{x_1}{x_1^2+x_2^2}\chi_1'(\theta)R_1^i.
	\end{equation*}
	These derivatives are continuous up to the corner, and they vanish on the corner,
	because $R_1^i=O(x_1^{2+\alpha})$ and $\partial_{x_1}R_1^i=O(x_1^{1+\alpha})$
	(actually, we only use $R_1^i=O(x_1^{1+\alpha})$ and $\partial_{x_1}R_1^i=O(x_1^\alpha)$ in this step).
	The continuity of the $y$-derivatives are obvious.
	Next, we consider the second derivatives. For example,
	\begin{multline}
		\label{eq:second-derivative-in-x1}
		\partial_{x_1}^2(\chi_1(\theta)R_1^i)
		=\frac{x_2^2}{(x_1^2+x_2^2)^2}\chi_1''(\theta)R_1^i
		+\frac{2x_1x_2}{(x_1^2+x_2^2)^2}\chi_1'(\theta)R_1^i \\
		\phantom{=\;}\,+2\cdot \frac{-x_2}{x_1^2+x_2^2}\chi_1'(\theta)\partial_{x_1}R_1^i
		+\chi_1(\theta)\partial_{x_1}^2R_1^i,
	\end{multline}
	and this is continuous up to the corner, and it vanishes on the corner,
	because $\partial_{x_1}^kR_1^i=O(x_1^{2-k+\alpha})$ for $k=0$, $1$, $2$.
	To summarize, we have shown that $\chi_1(\theta(x_1,x_2))R^i_1(x_1,y_1,y_2)$ is $C^2$
	in $\overline{U_1}\times\overline{U_2}$.

	It remains to show that
	$\partial_{x_1}^2(\chi_1(\theta)R_1^i)\in C^\alpha(\overline{U}_1\times\overline{U}_2)$.
	We shall prove that
	\begin{equation}
		I(x_1,y_1,x_2,y_2)=\frac{x_2^2}{(x_1^2+x_2^2)^2}\chi_1''(\theta)R_1^i(x_1,y_1,y_2)
	\end{equation}
	is in $C^\alpha(\overline{U}_1\times\overline{U}_2)$,
	as the rest of the terms on the right-hand side of \eqref{eq:second-derivative-in-x1} can be treated similarly.
	Let us write $R^i_1(x_1,y_1,y_2)=x_1^2S_1^i(x_1,y_1,y_2)$;
	then $S_1^i$ is in $C^\alpha(\overline{U}_1\times\overline{U}_2)$ and is $O(x_1^\alpha)$.
	We can express
	\begin{equation}
		I(x_1,y_1,x_2,y_2)=f(\theta)S_1^i(x_1,y_1,y_2),
	\end{equation}
	where $f\in C^\infty([0,\pi/2])$.
	Consequently, for $(x_1,y_1,x_2,y_2)$, $(x'_1,y'_1,x'_2,y'_2)\in\overline{U}_1\times\overline{U}_2$
	satisfying $x_2/x_1=x'_2/x'_1$, we have
	\begin{equation}
		\begin{split}
			\abs{I(x_1,y_1,x_2,y_2)-I(x'_1,y'_1,x'_2,y'_2)}
			&=
			\abs{f(\theta)}\abs{S_1^i(x_1,y_1,y_2)-S_1^i(x'_1,y'_1,y'_2)}\\
			&\leqq
			C\abs{(x_1,y_1,y_2)-(x'_1,y'_1,y'_2)}^\alpha.
		\end{split}
	\end{equation}
	On the other hand,	
	\begin{equation}
		\begin{split}
			&\abs{I(r\cos(\theta+\delta),y_1,r\sin(\theta+\delta),y_2)-I(r\cos\theta,y_1,r\sin\theta,y_2)}\\
			&\leqq\abs{f(\theta+\delta)-f(\theta)}\abs{S_1^i(r\cos\theta,y_1,y_2)}
			+\abs{f(\theta+\delta)}\abs{S_1^i(r\cos(\theta+\delta),y_1,y_2)-S_1^i(r\cos\theta,y_1,y_2)}\\
			&\leqq
			Cr^\alpha\abs{\delta}^\alpha.
		\end{split}
	\end{equation}
	We obtain $I(x_1,y_1,x_2,y_2)\in C^\alpha(\overline{U}_1\times\overline{U}_2)$ by combining these two
	estimates.
\end{proof}

\section{Proof of the main theorem}
\label{sec:proof-main}

To complete the proof of Theorem \ref{thm:main},
we use results from Li--Tam \cite{Li-Tam-91} on the heat equation
\begin{equation}
	\label{eq:heat-equation}
	\frac{\partial u(\orddot,t)}{\partial t}=\tau(u(\orddot,t))
\end{equation}
for harmonic maps.

\begin{proof}[Proof of Theorem \ref{thm:main}]
	Take an approximate harmonic map constructed in Proposition \ref{prop:approx-solution},
	which we write $v$ instead of $u$.
	Then, it follows from \cite{Li-Tam-91}*{Theorem 4.2} (which is on p.\ 25, misprinted as ``Theorem 4.1'') that
	\eqref{eq:heat-equation} has a unique solution with initial value $u(\orddot,0)=v$
	such that $\sup_{M\times[0,T]}e(u(\orddot,t))<\infty$ for all $T>0$.

	As $\sigma(X)=\sigma(\mathbb{H}^2\times\mathbb{H}^2)=\overline{\sigma(\mathbb{H}^2)+\sigma(\mathbb{H}^2)}$
	and $\sigma(\mathbb{H}^2)=[1/4,\infty)$, the spectrum of $X$ has a positive lower bound.
	Moreover, since $\abs{\tau(v)}=O(\rho_1^{\alpha/2}\rho_2^{\alpha/2})$ for any $\alpha\in(0,1)$,
	$\abs{\tau(v)}^2$ is in $L^p(X)$ for any $p>1$.
	In view of these facts, we may invoke \cite{Li-Tam-91}*{Theorem 5.2}, or the arguments in the proof thereof.
	It follows that
	\begin{equation}
		\Abs{\frac{\partial u(\orddot,t)}{\partial t}}\leqq Ce^{-at}\quad\text{in $X$}
	\end{equation}
	for some $a>0$ and $C>0$,
	and that $u(\orddot,t)$ converges uniformly in $X$ as $t\to\infty$ to a harmonic map $u_\infty$.
	Since the tension field of $v$ uniformly tends to zero at $\bdry X$, one can prove that
	\begin{equation}
		\label{eq:decay-of-time-derivative-at-topological-boundary}
		\lim_{p\to\bdry X}\sup_{t\in[0,T]}\Abs{\frac{\partial u(p,t)}{\partial t}}=0\quad\text{uniformly along $\bdry X$}
	\end{equation}
	for any $T>0$ and that
	\begin{equation}
		\label{eq:decay-of-distance-of-approximation-and-genuine-solution}
		\lim_{p\to\bdry X}d_{\mathbb{H}^{n+1}}(u_\infty(p),v(p))=0\quad\text{uniformly along $\bdry X$.}
	\end{equation}
	Therefore, the Euclidean distance between $u_\infty(p)$ and $v(p)$ also converges to zero as $p\to\bdry X$,
	which implies that $u_\infty$ extends to a continuous map $\overline{X}\to\overline{B^{n+1}}$
	and that $u_\infty|_{\bdryX}=v|_{\bdryX}=\tilde{\varphi}$.
	It follows from the latter that $u_\infty|_{\cornerX}=\varphi$ and
	$u_\infty(\overline{X}\setminus\cornerX)\subset B^{n+1}$.
\end{proof}

Finally, we remark that the estimate \eqref{eq:decay-of-distance-of-approximation-and-genuine-solution}
can be improved by following the proof of \cite{Li-Tam-93-I}*{Lemma 3.2}.
Let us first observe that, if we define $\rho\colon D\to\mathbb{R}$ by $\rho(p)=1-d_{\mathrm{Euc}}(0,p)^2$
and set $\rho_\lambda=\rho\compose\pi_\lambda$, $\pi_\lambda\colon X=D\times D\to D$ being the projection
to the $\lambda$-th factor, then a straightforward computation shows that
\begin{equation}
	\Delta_g(\rho_1^s\rho_2^s)\leqq 0,\qquad 0\leqq s\leqq 1,
\end{equation}
where $g$ is the product metric of the Poincar\'e metrics.
On the other hand, since the target has negative sectional curvature,
a result of Hartman \cite{Hartman-67} implies that
\begin{equation}
	\left(\Delta_g-\frac{\partial}{\partial t}\right)\left(\Abs{\frac{\partial u}{\partial t}}^2\right)\geqq 0.
\end{equation}
Consequently, for any $\alpha\in(0,1)$ and any $C'>0$, we have
\begin{equation}
	\left(\Delta_g-\frac{\partial}{\partial t}\right)
	\left(\Abs{\frac{\partial u}{\partial t}}^2-C'\rho_1^\alpha\rho_2^\alpha\right)\geqq 0.
\end{equation}

Let $\alpha$ be arbitrarily fixed, and take $\alpha'=\alpha+\varepsilon<1$, $\varepsilon>0$.
Then, since $\tau(v)=O(\rho_1^{\alpha'/2}\rho_2^{\alpha'/2})$ near $\bdryX$,
if we take sufficiently large $C'>0$, we have
\begin{equation}
	\label{eq:boundary-value-estimate-for-heat-flow-maximum-principle}
	\Abs{\frac{\partial u}{\partial t}}^2-C'\rho_1^{\alpha'}\rho_2^{\alpha'}\leqq 0
\end{equation}
at $t=0$. This fact, \eqref{eq:decay-of-time-derivative-at-topological-boundary},
and the maximum principle together imply that
\eqref{eq:boundary-value-estimate-for-heat-flow-maximum-principle} holds in $X\times[0,\infty)$.
Note that
\begin{equation}
	d_{\mathbb{H}^{n+1}}(u_\infty(\orddot),v(\orddot))
	\leqq\int_0^\infty\Abs{\frac{\partial u}{\partial t}}dt
	\leqq\int_0^T\Abs{\frac{\partial u}{\partial t}}dt
	+\int_T^\infty Ce^{-at}dt
	=\int_0^T\Abs{\frac{\partial u}{\partial t}}dt+\frac{C}{a}e^{-aT}
\end{equation}
for any $T>0$.
Based on this, for any $p\in X$, we can choose $T$ so that $e^{-aT}=\rho_1(p)^{\alpha'/2}\rho_2(p)^{\alpha'/2}$
to conclude
\begin{equation}
	d_{\mathbb{H}^{n+1}}(u_\infty(p),v(p))
	\leqq C''\rho_1(p)^{\alpha'/2}\rho_2(p)^{\alpha'/2}\log\left(\frac{1}{\rho_1(p)\rho_2(p)}\right),
\end{equation}
where $C''>0$ is independent of $p$.
Therefore, we obtain that
$d_{\mathbb{H}^{n+1}}(u_\infty(\orddot),v(\orddot))=O(\rho_1^{\alpha/2}\rho_2^{\alpha/2})$ near $\bdryX$
for any $\alpha\in(0,1)$.

\bibliography{myrefs}

\end{document}